\documentclass{article}

\usepackage[preprint]{neurips_2020}

\usepackage[utf8]{inputenc} % allow utf-8 input
\usepackage[T1]{fontenc}    % use 8-bit T1 fonts
\usepackage{hyperref}       % hyperlinks
\usepackage{url}            % simple URL typesetting
\usepackage{booktabs}       % professional-quality tables
\usepackage{nicefrac}       % compact symbols for 1/2, etc.
\usepackage{microtype}      % microtypography

\usepackage[textsize=tiny]{todonotes}

\usepackage{amsfonts, amsmath, amssymb, amsthm, mathrsfs}

\usepackage[capitalize,noabbrev]{cleveref}

\newtheorem{thm}{Theorem}
\newtheorem{corollary}{Corollary}
\newtheorem{defi}{Definition}
\newtheorem{lem}{Lemma}

\newtheorem{rmk}{Remark}

\newcommand\norm[1]{\lVert#1\rVert}

\newcommand\set[1]{\left\{ #1 \right\}}

\newcommand\innerprod[1]{\langle #1\rangle}

\newcommand\E{\mathbb{E}}
\newcommand\Prb{\mathbb{P}}

\newcommand \Lip{\mathbf{Lip}}

\DeclareMathOperator{\supp}{supp}

\title{Convex function through Doob-Meyer decomposition}

\author{%
  Minh Nguyen\thanks{Correspondence to: minhpnguyen@utexas.edu}\\
  Department of Mathematics\\
  University of Texas at Austin
}

\begin{document}

\maketitle

\begin{abstract}
In this work, we aim to study a strong version of Ito's lemma for convex function. By considering the corresponding sub-martingale on a Brownian motion, we gain more insights about the convex function through a probabilistic viewpoint. The Doob-Meyer decomposition of this sub-martingale subsequently helps us deduce the Ito's lemma for convex function, and enables us to study a convex function via stochastic calculus. In particular, we use this version of Ito's lemma together probabilistic inequalities to recover an important analytic property of the convex function, which is its second-order differentiability.
\end{abstract}

We start by outlining a general plan for the paper. We first introduce the Doob-Meyer decomposition of a sub-martingale, which is obtained by applying a convex function $f$ on a Brownian motion. We study the decomposition's properties, and the connection between its Revuz measure and the second derivative of $f$. We then use the decomposition to prove a strong Ito's lemma for convex function $f$ in the next section. We end the paper by applying this Ito's lemma to the \textbf{Busemann-Feller-Alexandrov} theorem for a convex function. 
\vspace{2mm}

Throughout this paper, let $(\Omega, \mathcal{F}, \Prb)$ be a fixed probability space with a filtration $\set{\mathcal{F}_t}$. We denote $W_t^x$ is the $d$-dimensional Brownian motion starting at $x$ with respect to this filtration. Then we have $W_t^x = x + W_t^0 = x+ W_t$. 
\vspace{2mm}

\section{Doob-Meyer decomposition of submartingale}
We now introduce the Doob-Meyer decomposition of sub-martingale and then characterize the non-decreasing process component $A_t$ by the Revuz measure.
\begin{lem}\label{doobmeyer_of_f}
For the convex function $f$ with at most polynomial growth, the process $f(W_t^x)$ is a sub-martingale, and the Doob-Meyer decomposition of $f(W_t^x)$ is $f(x) + M^f_t + A^f_t = M_t + A_t$, where $A_t$ is a non-decreasing process and is a positive continuous additive functional (PCAF) of $W$.
\end{lem}
\begin{proof}
WLOG, we assume that $f$ is at most linear growth: $f(x) \leq A\norm{x}+B$. Because $f$ is convex function, $f(W_t^x)$ is a local sub-martingale. We focus on proving $f(W_t^x)$ is a proper sub-martingale. To do so, we only need to prove that the process $f(W_t^x)$ is in (DL) class. In other word, for a fixed $t$, we prove that $\set{f(W_{\tau}^x)| \text{ stopping time } \tau \leq t}$ is uniformly integrable. Now since $f(W_{\tau}^x) \leq (AW_{\tau}+B)^2$, we have $$\E^x[f(W^x_{\tau})^2] \leq \E[(AW^x_{\tau}+B)^2] \leq C_1 t+ C_2 < \infty$$
Then two conditions of uniform integrability follow easily by Cauchy-Schwarz theorem. 
\end{proof}
\vspace{2mm}

We now introduce Revuz measure and weak second-derivative of the convex function $f$. At the end of this section, we will show a relation between these objects (see \cref{trace_lem}). 
\vspace{2mm}

\begin{defi} (Revuz measure)
Suppose that we are given a PCAF $A$ of $W$. Let $m$ be the Lebesgue measure, and $(f.A)_t = \displaystyle\int_0^t f(W_s)dA_s$. We define the Revuz measure of $A$ as
$$ v_A(f) = \lim_{t \downarrow 0} \frac{1}{t}\E^m[(f.A)_t]$$
\end{defi}
\vspace{2mm}

\begin{rmk}
The Revuz measure $v_A$ is also a Radon measure. Moreover, since $m$ is invariant measure of $W$,  $v_A(f) = \E^m[(f.A)_1]$. Thus, $v_A$ is a $\sigma$-finite Radon measure on $\mathbb{R}^d$.
\end{rmk}
\vspace{2mm}

\begin{defi} (Weak second-derivative of convex function)
Given the convex function $f$, let $p(x)$ be a choice of the sub-gradient. Define the linear functional $T_{ij}$ on $C_c^{\infty}(\mathbb{R}^d)$ so that $(T_{ij}, \varphi):= (f, (\varphi)_{ij})$. Then $T_{ij}$ is a positive linear functional and is uniquely determined by a Radon measure $\mu_{ij}$ on $\mathbb{R}^d$ such that $(T_{ij}, \varphi) = \int \varphi \mu_{ij}(dx)$. By packing $\mu_{ij}$ together, we get the matrix-valued measure $\mu$. We called $\mu$ the second derivative measure of $f$.
\end{defi}

\begin{defi}
Let $Q(x)$ be the Radon-Nikodym derivative of the second derivative Radon measure $\mu$: $Q(x)_{ij} = d(\mu_{ij})/dx$. By an approximation argument, we can show that $Q(x)$ is non-negative definite for a.e $x$.
\end{defi}

We use the notations $f, \mu, p, Q$ in the subsequent \cref{EAtdividet}, \cref{chain_rule}, \cref{directional_derivative_exist}, \cref{ito_lemma}. From now on, we assume that $f$ has a polynomial growth. To estimate $f(W_t)$ and $f(y)$ in \cref{directional_derivative_exist} and \cref{ito_lemma}, we assume the linear growth on $f$, but our estimate is still valid for polynomial growth case. Next we present an important theorem that represents expectation of the PCAF component of $f(W_t^x)$ in terms of the Revuz measure $v_A$. A direct consequence of this theorem is \cref{trace_lem}. 

\begin{thm}(Representation theorem) \label{rep_thm}
Let $A$ be a PCAF of the Brownian motion $W$ with respect to the convex function $f$ as describe above. Let $\innerprod{v, g} = \int g(x)v(dx)$ for a function $g$ and measure $v$, and $(., .)$ be the usual inner product between $2$ functions. Then we have the following identity for any $C^2$ function $h$ with compact support:
\begin{align}\label{rep_id_rev}
(h, \E^{x}[A_t]) = \int_{0}^t \innerprod{v_A, p_sh}ds
\end{align}
\end{thm}

\begin{proof}
It is enough to prove \eqref{rep_id_rev} for nonnegative $h\in C_c^\infty$, since the general $C_c^2$ case then follows by linearity and density. As in the sequel, we first treat the linear-growth case $f(x)\le C(1+\|x\|)$; the polynomial-growth case is obtained by the same truncation argument.

Let $f_n=f*\phi_{\varepsilon_n}$ be the standard convex mollifications, with $\varepsilon_n\downarrow0$. Then $f_n\in C^\infty$, $f_n$ is convex, $f_n\to f$ locally uniformly, and for $n$ large, $f_n(x)\le C(1+\|x\|)$. Write:

\[
f_n(W_t^x)=f_n(x)+M_t^n+A_t^n,\qquad
f(W_t^x)=f(x)+M_t+A_t
\]

for the corresponding Doob--Meyer decompositions. For smooth $f_n$, It\^o's formula gives us:

\[
A_t^n=\frac12\int_0^t \Delta f_n(W_s^x)\,ds,
\qquad 
v_{A^n}(dy)=\frac12 \Delta f_n(y)\,dy .
\]

Hence, if
\[
g_t(y):=\int_0^t p_sh(y)\,ds,
\]
Then by Fubini and symmetry of the Brownian semigroup, we have:
\begin{align}
(h,\E^x[A_t^n])
&=\frac12\int_{\mathbb R^d} h(x)\int_0^t \E^x[\Delta f_n(W_s)]\,ds\,dx \notag\\
&=\frac12\int_0^t\int_{\mathbb R^d} \Delta f_n(y)\,p_sh(y)\,dy\,ds \notag\\
&=\int_0^t \langle v_{A^n},p_sh\rangle\,ds
= v_{A^n}(g_t). \label{eq:smooth_rep}
\end{align}

We next pass to the limit $n\to\infty$. First, on every compact set $K$,
\[
\E^x[A_t^n]-\E^x[A_t]
=\E^x[f_n(W_t)-f(W_t)]-(f_n(x)-f(x)).
\]

Using the local uniform convergence $f_n\to f$, the common linear-growth bound, and a standard decomposition of the domain into $\{W_t\in B_R\}$ and $\{W_t\notin B_R\}$, one obtains
\[
\sup_{x\in K}\big|\E^x[A_t^n]-\E^x[A_t]\big|\to0.
\]
In particular, since $\supp(h)$ is compact,
\begin{equation}\label{eq:l1lhsconv}
(h,\E^x[A_t^n])\to (h,\E^x[A_t]).
\end{equation}

Second, by applying the stability of convex semimartingale decompositions under locally uniform approximation (see \cite{carlen_protter_1992} for proof details) to $W_{t\wedge1}$, and then localizing by the exit time $T_R=\inf\{s:\|W_s^x\|>R\}$, we get:
\[
\E^x\!\left[\sup_{s\le 1\wedge T_R}|A_s^n-A_s|\right]\to0.
\]

Combining this with integration by parts shows that for every nonnegative bounded $g\in C^2$ with bounded first and second derivatives, and every compact $D$,
\begin{equation}\label{eq:gAconv}
\int_D \E^x[(g.A^n)_1]\,dx \to \int_D \E^x[(g.A)_1]\,dx .
\end{equation}

We now prove the two inequalities.

\medskip
\noindent\emph{Step 1:}
\[
(h,\E^x[A_t])\ge \int_0^t \langle v_A,p_sh\rangle\,ds .
\]
Since $h\ge0$, also $g_t\ge0$. For every compact $D$,
\[
v_{A^n}(g_t)=\E^m[(g_t.A^n)_1]\ge \int_D \E^x[(g_t.A^n)_1]\,dx .
\]
Using \eqref{eq:smooth_rep}, \eqref{eq:l1lhsconv}, and then \eqref{eq:gAconv}, we obtain:
\[
(h,\E^x[A_t])
=\lim_{n\to\infty} v_{A^n}(g_t)
\ge \int_D \E^x[(g_t.A)_1]\,dx .
\]
Letting $D\uparrow\mathbb R^d$ and using monotone convergence, we get:
\[
(h,\E^x[A_t])\ge \E^m[(g_t.A)_1]=v_A(g_t)
=\int_0^t \langle v_A,p_sh\rangle\,ds .
\]

\medskip
\noindent\emph{Step 2:}
\[
\int_0^t \langle v_A,p_sh\rangle\,ds \ge (h,\E^x[A_t]).
\]

Fix a ball $D$ containing $\supp(h)$, choose $l_D\in C_c^\infty$ with
$0\le l_D\le1$ and $l_D\equiv1$ on $D$, and set
\[
g_{t,D}:=l_D g_t\in C_c^\infty .
\]
For smooth $f_n$, since $v_{A^n}(dy)=\frac12\Delta f_n(y)\,dy$, Fubini and It\^o's formula up to the exit time $\tau_D$ from $D$ yield:
\[
v_{A^n}(g_{t,D})
\ge (h,\E^x[f_n(W_{t\wedge\tau_D})]-f_n(x)).
\]
Moreover, $g_{t,D}$ has compact support, so by the same localization argument behind \eqref{eq:gAconv},
\[
v_A(g_{t,D})\ge \limsup_{n\to\infty} v_{A^n}(g_{t,D}).
\]
Passing to the limit and using $f_n\to f$ uniformly on $D$, we have:
\[
v_A(g_t)\ge v_A(g_{t,D})
\ge (h,\E^x[f(W_{t\wedge\tau_D})]-f(x)).
\]

Finally, because $D\supset \supp(h)$ and $f$ has linear growth,
\[
(h,\E^x[f(W_t)-f(W_{t\wedge\tau_D})])\to0
\qquad \text{as } D\uparrow\mathbb R^d,
\]
by a standard Gaussian-tail/Doob-maximal estimate for Brownian motion. As a result,

\[
v_A(g_t)\ge (h,\E^x[f(W_t)]-f(x))=(h,\E^x[A_t]).
\]

Combining the two steps gives
\[
(h,\E^x[A_t])=v_A(g_t)
=\int_0^t \langle v_A,p_sh\rangle\,ds ,
\]
which is \eqref{rep_id_rev}.
\end{proof}

\begin{corollary}\label{EAtdividet_formula}
For a.e $x$, we have:
$$\frac{\E^x[A_t]}{t} = \frac{1}{t}\int_0^t \int_{\mathbb{R}^d}\frac{1}{(2\pi s)^{d/2}}e^{-\norm{x-y}^2/2s}v_A(dy)ds$$
\end{corollary}
\vspace{2mm}

Before proving the connection between the Revuz measure and the (weak) second derivative of $f$, we need the following technical lemma:
\begin{lem}\label{sato_lem}(A special case of Van der Vaart's theorem)
For a measure $\mu$ and a function $\psi(x) = (2\pi)^{-d/2} e^{-\norm{x}^2/2}$ on $\mathbb{R}^d$, denote $\psi_T(x) = T^d \psi (Tx)$. We have:

$$\lim_{T \to \infty} \int_{\mathbb{R}^d} \psi_T(x-y)\mu(dy) = \frac{d \mu}{dm}(x)$$
for a.e $x$. Here $\frac{d \mu}{dm}$ is the Radon-Nikodym derivative of measure $mu$ with respect to Lebesgue measure $m$.
\end{lem}
\begin{proof}
See \cite{bourgain_sato_1986} for a detailed proof.
\end{proof}
\vspace{2mm}

We're now ready to prove \cref{trace_lem}, which further helps us relate the analytic proprieties of $f$ with its probabilistic counterpart.
\begin{lem} (Revuz-Trace lemma)\label{trace_lem}
Let $A_t$ to be the PCAF part of Doob-Meyer decomposition of $f(X_t)$. Then the Revuz measure $v_A$ of $A$ coincides with half of the trace $\frac{1}{2}\sum_{i=1}^d \mu_{ii}$ of the second derivative $\mu$ of $f$.
\end{lem}
\begin{proof}
Consider non-negative $C^{\infty}$ function $h$ with compact support. First we have:
\begin{align*}
\lim_{t \to 0}\frac{1}{t}(f, p_t h - h) &= \left(f, \lim_{t \to 0}\frac{p_t h - h}{t} \right) = (f, \frac{1}{2}\Delta h)\\
&= \frac{1}{2}\sum_{i=1}^n (f, h_{ii}) = \frac{1}{2}\sum_{i = 1}^n \innerprod{h, \mu_{ii}} = \frac{1}{2}\innerprod{h, tr(\mu)}
\end{align*}
Moreover, by \cref{rep_thm}, for any $h \in C_c^{\infty} \geq 0$
\begin{align}\label{inter_iden_revuz_trace}
\begin{split}
&(f, p_th-h) = (f, p_th) - (f, h) = (p_tf, h) - (f, h) = (p_tf-f, h)\\ 
&= (\E^x[f(W_t)]-f(x), h) = (\E^x[A_t], h) = \int_0^t\innerprod{v_A, p_sh} ds = \innerprod{v_A, \int_0^t p_shds}
\end{split}
\end{align}
For every $h \in C^{\infty}$ so that $(f, h)$ and $(f, p_t h)$ are finite, \cref{inter_iden_revuz_trace} also holds for $h$.

$$\lim\limits_{t \to 0} \dfrac{(f, p_th-h)}{t} = \lim_{t \to 0}\ \innerprod{v_A, \frac{1}{t}\int_0^t p_sh ds} = \innerprod{v_A, h}$$

As a result, $\innerprod{h, v_A} = \frac{1}{2}\innerprod{h, tr(\mu)}$ for all non-negative $C_c^{\infty}$ function $h$. By linearity, $\innerprod{h, v_A} = \frac{1}{2}\innerprod{h, tr(\mu)}$ for any $C_c^{\infty}$ function $h$. Thus, as functionals, $v_A$ and $\frac{1}{2}tr(\mu)$ are coincide on a dense set of its domain and, therefore, must be equal.
\end{proof}
\section{Ito's lemma and stochastic calculus for convex function}
In this section, given the convex function $f$, we use the previous Doob-Meyer decomposition of $f(W_t)$ to prove a version of Ito's lemma for convex function. We need to prove a couple of technical lemmas including \cref{EAtdividet}, \cref{chain_rule}, and \cref{directional_derivative_exist} in order to prove the Ito's lemma in \cref{ito_lemma}
\begin{lem}\label{EAtdividet}
For a.e $x$, we have the following identity:
$$\lim_{t \to 0} \frac{1}{t}\E[|f(W_t^x) -f(x) - \innerprod{p(x),W_t}|] = \frac{1}{2}tr(Q(x))$$
\end{lem}
\begin{proof}
Because $|f(W_t^x) -f(x) - p(x) W_t| = f(W_t^x) -f(x) - p(x) W_t \geq 0$, and $\E[\innerprod{p,W_t}] = 0$ for any vector $p$, we have:	
$$\lim_{t \to 0} \frac{1}{t}\E[|f(W_t^x) -f(x) - p(x) W_t|] = \lim_{t \to 0} \frac{1}{t} \E[f(W_t^x) -f(x)] = \lim_{t \to 0} \frac{\E[A_t]}{t}$$ 
By \cref{EAtdividet_formula} and \cref{sato_lem} for $v_A$ and $T = 1/\sqrt{t}$, 
\begin{align*}
\lim_{t \to 0} \frac{\E[A_t]}{t} &= \lim_{t \to 0} \frac{1}{t}\int_0^t \int_{\mathbb{R}^d}\frac{1}{(2\pi s)^{d/2}}e^{-\norm{x-y}^2/2s}v_A(dy)ds\\ 
&= \lim_{t \to 0}\int_{\mathbb{R}^d} \frac{1}{(2\pi t)^{d/2}}e^{-\norm{x-y}^2/2t}v_A(dy) = \frac{dv_A}{dm}(x)
\end{align*} 
By \cref{trace_lem}, the Revuz measure $v_A$ is exactly $\frac{1}{2} tr(\mu)$. Thus, $\frac{dv_A}{dm}(x) = \frac{1}{2} tr(Q(x))$ for a.e $x$. Combining with the two equalities above, we can finish our proof here.
\end{proof}
\vspace{2mm}

\begin{rmk}
Because the Revuz measure $v_A$ is $\frac{1}{2} tr(\mu)$ so that $\frac{1}{2} tr(Q(x))$ is in fact the Radon-Nikodym derivative of $v_A$. Moreover, by using \cref{EAtdividet}, and by splitting $Q$ into $Q = Q^+- Q^-$, we can easily prove Alexandrov theorem in 1D.
\end{rmk}
\vspace{2mm}

\begin{lem}\label{chain_rule}
Given an invertible matrix $S$, for a.e $x$, we have:
$$\lim_{t \to 0} \frac{\E[f(SW_t + x)]-f(x)}{t} = \lim_{t \to 0} \frac{\E[f(S(W_t+S^{-1}x )] - f(S(S^{-1}x))}{t} = \frac{1}{2}tr(S^TQ(x)S)$$
\end{lem}
\begin{proof}
First we show that for an invertible matrix $S$, if the second derivative Radon measure $\mu^S$ of $f(Sx)$ has the Radon-Nikodym derivative, $Q^S(x)$, then for a.e $x$,
\begin{equation}\label{density_eqn}
Q^S(x) = S^TQ(S(x)) S
\end{equation}
Take any function $\varphi \in C_c^{\infty}(\mathbb{R}^d)$, $\forall i, j \in \overline{1, n}$, we have:
$$\int \varphi(Sx)\mu^S_{ij}(dx) = \int f(Sx)(\varphi(Sx))_{ij}(x)dx= \int f(Sx)(S^TD^2\varphi S)_{ij}(Sx)dx$$
Putting these $d^2$ equations into $d \times d$ matrix, we get:
\begin{equation*}
\int\varphi(Sx)\mu^S(dx) = \frac{S^T}{|\det S|} \int f(y)D^2\varphi(y)dy S = \frac{S^T}{|\det S|}\int \varphi(y)\mu(dy)S
\end{equation*}
Since this identity is true for any smooth function, by bounded convergence theorem, it is also true for any bounded function with compact support. We can take $\varphi = 1_A$ for any bounded Borel set $A$ so that:
\begin{equation}\label{measure}
\mu^S(S^{-1}A) = \frac{S^T}{\det S}\mu(A)S
\end{equation}
Now consider any bounded Borel set $A$, which has no intersection with the singular support of $\mu$, and the rotation of the singular support of $\mu^S$ by a matrix $S$. The latter means that $S^{-1}A$ has no intersection with the singular support of $\mu^S$. Then \cref{measure} implies:
\begin{equation}\label{density_int_equal}
\int_{S^{-1}A} Q^S(x)dx = \frac{S^T}{\det S} \int_A Q(x)dx S =  S^T \int_{S^{-1}A} Q(y)dy S
\end{equation}
Note that \cref{density_int_equal} holds for any set $A$ that has no intersection with two null sets. Moreover, we can add to $A$ any set of measure zero so that \cref{density_int_equal} still holds. Therefore, \cref{density_int_equal} indeed holds for any bounded Lebesgue measurable set $A$. As a result, \cref{density_eqn} holds for a.e $x$, and we finish the proof by invoking \cref{EAtdividet}.
\end{proof}
\vspace{2mm}

\begin{lem}\label{directional_derivative_exist}
For each unit vector $v$, for a.e $x$ the following limit is $0$: 
$$\lim_{r \to 0}\frac{1}{r^2}|f(x + rv)-f(x)-p(x).rv - \frac{1}{2}\innerprod{Q(x)v, v}r^2| = 0$$
\end{lem}
\begin{proof}
Note that we can remove the absolute sign of the expression inside the limit because the expression is a function of $r$, which is only a real number. Consider an unit vector $v$. WLOG, assume that $v = e_1 = (1, 0, \cdots, 0)$. Let $l_{x_0}$ be the line $\set{rv + x_0}_{r \in \mathbb{R}}$. If $p(x)$ is sub-gradient of $f$ at $x$, $p(x)v$ is sub-gradient of $f_1(r) = f(x_0 + rv)$ at $r = (x-x_0)\cdot v$. Therefore, for each $x_0=(0, y_0) \in \set{0} \times \mathbb{R}^{d-1}$, by 1D Alexandrov theorem on $f_1$, the set $E^{y_0}$ = $\{r \in \mathbb{R}$ so that $x = (r, y_0) \in l_{x_0}: f_1$  is not second $v$-directional differentiable at $x$ with first derivative $p(x)\}$ has 1D Lebesgue measure zero. Let $\mathcal{S}$ be the set of all $x \in \mathbb{R}^d$ so that there is some $q = q(x)$ on $\mathbb{R}^d$ such that:
$$\lim_{r \to 0} \frac{f(x+rv) - f(x) - rp(x)v}{r^2} = \frac{1}{2}q(x)$$

Because $f$ is continuous, we can consider the limit in LHS only over the rational number $r$. Each function $\frac{f(x+rv) - f(x) - rp(x)v}{r^2}$ of variable $x$ is measurable, and so is the set of $x$ on which the limit of the sequence of these measurable functions exist. Thus, $\mathcal{S}$ is Lebesgue measurable. Because $\mathcal{S}^c \subset \bigcup_{y \in \mathbb{R}^{d-1}}(E^y \times \set{y})$, $m(\mathcal{S}^c) \leq \int_{\mathbb{R}^{d-1}} m(E^y)dy = 0$. The integral is well-defined because of the Fubini's theorem. Hence, $m(\mathcal{S}^c) = 0$. In the general case, when $v \neq e_1$, we integrate over the $d-1$-dimensional subspace with normal being vector $v$ instead of $\mathbb{R}^{d-1}$.

Let $B_t$ be $1D$ Brownian motion (we use a different notation $B$ instead of $W$ to emphasize the dimension). Then for $x \in \mathcal{S}$, by a simple estimation, we can show that
$$\lim_{t \to 0}\frac{\E[f(x+ B_tv) - f(x) - B_t p(x)v]}{t} = \frac{1}{2}q(x)$$

Let $R$ to be the rotation matrix that rotates the basis vector $e_1 = (1, 0, \cdots, 0) \in \mathbb{R}^d$ to $v$, and $S_1$ to be the matrix with a single non-zero entry $(S_1)_{11} = 1$. Finally, let $S = RS_1$, and $S^+_{\epsilon} = R(S_1 + \epsilon I_n) = RS_1 + \epsilon R$, and $S^-_{\epsilon} = R(S_1 - \epsilon I_n) = RS_1 - \epsilon R$. Fix $x \in \mathcal{S}$, and define $g(y) = f(x+y)-f(x)-p(x)y$. Let $\alpha = 1/(1+\epsilon)^2$. Now, by the convexity of $g$, we have:
$$\E[g(S^+_{\epsilon}W_{\alpha t})] = \E\left[g\left(S^+_{\epsilon}\frac{W_t}{1+\epsilon}\right)\right] \leq \frac{1}{1+\epsilon}\E[g(RS_1W_t)] + \frac{\epsilon}{1+\epsilon} \E[g(RW_t)]$$

Thus, $$\E[g(SW_t)] = \E[g(RS_1W_t)] \geq (1+ \epsilon) \E[g(S^+_{\epsilon}W_{\alpha t})] - 
\epsilon\E[g(RW_t)]$$

Therefore, by \cref{chain_rule}, we have:
\begin{align*}
&\lim_{t \to 0}\frac{\E[g(SW_t)] }{t} \geq (1+ \epsilon) \lim_{t \to 0}\frac{\E[g(S^{+}_{\epsilon}W_{\alpha t})]}{t} - 
 \epsilon\lim_{t \to 0}\frac{\E[g(RW_t)]}{t} \\
 &= \frac{(1+ \epsilon)\alpha}{2} tr((S^+_{\epsilon})^TQ(x) S^+_{\epsilon})- \epsilon\lim_{t \to 0}\frac{\E[g(W_t)]}{t}\\
 \end{align*}
 The second last equality follows from the fact that Brownian motion is invariant under rotation. By letting $\epsilon \to 0$, we obtain the lower bound inequality:
 $$\lim_{t \to 0}\frac{\E[g(SW_t)] }{t} \geq  \frac{1}{2}tr(S^TQ(x)S)$$ Similarly, if we consider $S^-_{\epsilon}$ instead, we would get the upper bound inequality. Thus, by noting that the first coordinate of the Brownian motion is itself a Brownian motion, we get $$\lim_{t \to 0}\frac{\E[g(B_tv)] }{t} = \lim_{t \to 0}\frac{\E[g(SW_t)] }{t} =  \frac{1}{2}tr(S^TQ(x)S) =  \frac{1}{2}\innerprod{Q(x)v, v}$$

 Therefore, the second directional derivative of $f$ at $x \in \mathcal{S}$ in the direction $v$ is $(p(x), q(x)) = (p(x), \innerprod{Q(x)v, v})$ as desired. Because $m(\mathcal{S}^c) = 0$, we can finish the proof here.
\end{proof}
\vspace{2mm}

\begin{thm}(Ito-like lemma for convex function)\label{ito_lemma}
For a convex function $f$, for a.e $x$, we have: 
$$\lim_{t \to 0}\frac{1}{t}\E[|f(W_t^x) - f(x) - \innerprod{p(x), W_t} - \innerprod{Q(x)W_t, W_t}|]  = 0$$
\end{thm}
\begin{proof}
The unit sphere $S^{d-1}$ can be parameterized by $d-1$ parameters so that the domain of the parameterization is the $d-1$-dimensional hypercube. We consider a grid on this hypercube domain. For an $\epsilon > 0$, the grid size can be chosen small enough so that the surface area of a grid's image on the unit sphere $S^{d-1}$ is less than $\epsilon$. We map back grid points on the grid to points on the sphere $S^{d-1}$, and consider the rays starting from the origin and passing through those grid points. These rays divide the space $\mathbb{R}^d$ into narrow cones $C_1, \cdots, C_N$ ($N = N_{\epsilon}$ depends on $\epsilon$ only). We consider a sequence $\set{\epsilon_n} \to 0$, and the sequence of grids associated with $\epsilon = \epsilon_n$ in this sequence. We call the unit directions of rays of these grids \textit{rational} directions. There are only countably many of \textit{rational} directions. Applying \cref{directional_derivative_exist} to all rational directions, we get that for almost every $x$ the following statement is true: for any rational direction $v$, we have:
\begin{equation}\label{directional1}
\lim_{r \to 0}\frac{1}{r^2}|f(x + rv)-f(x)-p(x).rv - \frac{1}{2}\innerprod{Q(x)v, v}r^2| = 0
\end{equation}

Fix an $x$ that satisfies this directional derivative property. Fix an $\epsilon$ (in the sequence $\set{\epsilon_n}$). We can see that $Q(x)$ is definite non-negative, and, for notation convenience, let $Q(x) = Q$, and $p(x) = p$. WLOG, we can assume that $Q \neq 0$. If $Q = 0$, the absolute value is eliminated, and we can invoke \cref{EAtdividet} to finish the proof.

For a real-valued process $X_t$, $|X_t| = 2X_t^+ - X_t$. Thus, if $\lim\limits_{t \to 0} \dfrac{\E[X_t]}{t}$ exists, we must have:
$$\limsup_{t \to 0} \frac{\E[|X_t|]}{t} = 2\limsup_{t \to 0}\frac{\E[X_t^+]}{t} + \lim_{t \to 0}\frac{\E[X_t]}{t}$$

Let $X_t = f(W_t^x) - f(x) - \innerprod{p,W_t} - \frac{1}{2}\innerprod{QW_t, W_t}$. By \cref{EAtdividet}, we have:
\begin{align*}
\lim_{t \to 0}\frac{\E[X_t]}{t} &=\lim_{t \to 0}\E\left[\frac{f(W_t^x) - f(x) - \innerprod{p,W_t}}{t}\right] - \lim_{t \to 0}\frac{1}{t}\E\left[\frac{1}{2}\innerprod{QW_t, W_t}\right]\\
&= \frac{1}{2}tr(Q(x)t) -  \frac{1}{2}tr(Q(x)t) = 0
\end{align*}

Therefore, by letting $g(y) = f(y+x)-f(x)-\innerprod{p, y}- \frac{1}{2}\innerprod{Qy, y}$, we have:
\begin{align}\label{main_inequality}
\begin{split}
\limsup_{t \to 0} &\frac{\E[|f(W_t^x) - f(x) - \innerprod{p,W_t} - \frac{1}{2}\innerprod{QW_t, W_t}|]}{t}  = \limsup_{t \to 0}\frac{\E[|X_t|]}{t}\\ 
&= 2 \limsup_{t \to 0}\frac{\E[X_t^+]}{t} = 2\limsup_{t \to 0}  \frac{\E[g(W_t)^+]}{t}
\end{split}
\end{align}

If we replace $f$ by $f + \frac{1}{2}\norm{x}^2$, $f$ is still convex and at most polynomial growth, and the quotient inside the limit still the same. But $p(x)$ is replaced by $p(x)+x$, and $Q(x)$ is replaced by $Q(x) + I_n$. In this case, the new second derivative density $Q(x) \succeq I_n$ so that $\innerprod{Q(x)y, y} \geq \norm{y}^2$ for each $y$. Thus, WLOG, we can assume that for each $x$, $\innerprod{Q(x)y, y} \geq \norm{y}^2\ \forall y$.
\vspace{2mm}

Consider the $\epsilon$-grid described above. We now bound $g(W_t)$ in term of $\epsilon$. We define $2^{d-1} \times N$ maps $T^1_j, \cdots, T^N_j: \mathbb{R}^d \to \mathbb{R}^d$ for $j \in \overline{1, 2^{d-1}}$ as follows: for each $i \in \overline{1, N}$, if $y \in C_i$, for each $y$, consider the tangent hyperplane passing through $y$ to the sphere of radius $\norm{y}$. Each region $C_i$ is defines by $2^{d-1}$ rays from the origin. Let $y_1, \cdots, y_{2^{d-1}}$ be the intersections of these $2^{d-1}$ rays with the above tangent hyperplane, and define $T^i_j(y) = y_j$ for $j \in {1, 2^{d-1}}$. Otherwise, if $y \not\in C_i$, define $T^i_j(y) = 0$ for all $j \in \overline{1, 2^{d-1}}$.

As we let $\epsilon = \epsilon_n \to 0$, all the intersection regions of the cones $C_i$ with unit sphere $S^{d-1}$ shrink to one points, and so are $C_i$'s $2^{d-1}$ intersections with the tangent hyperplanes. The ratio $T^i_j(y)/ y$ doesn't depends on the norm of $y$. Moreover, as $\epsilon \to 0$, $T^i_j(y) \to y$ uniformly in $i, j$ and in $y \in C_i \cap S^{d-1}$. We also have the uniform convergence: $\innerprod{QT^i_j(y), T^i_j(y)} \to \innerprod{Qy, y}$ (in $i, j$, and $y \in C_i \cap S^{d-1}$). Therefore, there exists $a(\epsilon)$ with $a(\epsilon) \to 0$ as $\epsilon \to 0$ so that for $y \in C_i \neq 0$, we have:
\begin{equation}\label{cond1} 
\norm{T^i_j(y)} \leq (1 + a(\epsilon))\norm{y} \text{, and }
\end{equation}
\begin{equation}\label{cond2}
\left|\innerprod{QT^i_j(y), T^i_j(y)}-\innerprod{Qy, y} \right| \leq a(\epsilon)\norm{y}^2 \leq a(\epsilon)\innerprod{Qy, y}\ \forall j \in \overline{1, 2^{d-1}}
\end{equation}

Furthermore, if $y \in C_i$, then $y$ is a convex combination of the points $T^i_j(y)$ so that $y = \sum_{j = 1}^{2^{d-1}} \alpha^{i, y}_j T^i_j(y)$ for $0 \leq \alpha^{i, y}_j \leq 1$, and $\sum_{j=1}^{2^{d-1}} \alpha^{i, y}_j = 1$. To simplify notation, we drop the superscript $W_t$ in $\alpha^{i, W_t}_j = \alpha^i_j$. Now let $g_1(y) = f(y+x)-f(x)-py$. Then $g_1$ is a convex function and $g(y) = g_1(y) - \frac{1}{2}\innerprod{Qy, y}$. By using the convexity of $g_1$ (the first inequality), we have:
\begin{align*}
&g(W_t)= \sum_{i = 1}^N g(W_t)1_{\set{W_t \in C_i}} =
\sum_{i = 1}^N (g_1(W_t)1_{\set{W_t \in C_i}} - \frac{1}{2}\innerprod{QW_t, W_t}1_{\set{W_t \in C_i}})\\
&\leq  \sum_{i =1}^N \bigg(\sum_{j =1}^{2^{d-1}} \alpha^i_j g_1(T^i_j(W_t))1_{\set{W_t \in C_i}}- \frac{1}{2}\innerprod{QW_t, W_t}1_{\set{W_t \in C_i}}\bigg)\\
&\leq  \sum_{i =1}^N \bigg(\sum_{j =1}^{2^{d-1}} \alpha^i_j g(T^i_j(W_t))1_{\set{W_t \in C_i}}+ \frac{1}{2}\sum_{j =1}^{2^{d-1}}\alpha^i_j\left| \innerprod{QW_t, W_t} -\innerprod{QT^i_j(W_t), T^i_j(W_t)} \right|1_{\set{W_t \in C_i}} \bigg)
\end{align*}

By using the constraint \cref{cond2}, we get:
\begin{align*}
&g(W_t) \leq  \sum_{i =1}^N \left(\sum_{j =1}^{2^{d-1}} \alpha^i_j g(T^i_j(W_t))1_{\set{W_t \in C_i}} + \frac{1}{2}\sum_{j =1}^{2^{d-1}}\alpha^i_j a(\epsilon)\innerprod{QW_t, W_t}1_{\set{W_t \in C_i}} \right)\\
&= \sum_{i =1}^N \sum_{j =1}^{2^{d-1}} \alpha^i_j g(T^i_j(W_t))1_{\set{W_t \in C_i}}  + \frac{1}{2}a(\epsilon)\innerprod{QW_t, W_t} \left( \sum_{i =1}^n 1_{\set{W_t \in C_i}} \right)\\
&\leq \sum_{i =1}^N \sum_{j =1}^{2^{d-1}} \alpha^i_j g(T^i_j(W_t))1_{\set{W_t \in C_i}}  + \frac{1}{2}a(\epsilon)\norm{Q} \norm{W_t}^2
\end{align*}

We prove that 
\begin{equation}\label{t_eqn}
\lim_{t \to 0} \dfrac{\E[|g(T_i^j(W_t))|]}{t} = 0
\end{equation}\\

By the first constraint, $\norm{T_i^j(W_t)} \leq (1 + a(\epsilon))\norm{W_t}$. Moreover, because the direction of $T_i^j(W_t)$ is rational, for each $\epsilon_0 > 0$, by \cref{directional1}, there exists $\delta > 0$ so that $g(T_i^j(y)) < \epsilon_0 \norm{T_i^j(y)}^2$ for all $\norm{T_i^j(y)} < \delta$. Thus, for $\norm{W_t}  < \delta/(1 + a(\epsilon)) = \delta_0$, we surely have $g(T_i^j(W_t)) < \epsilon_0 \norm{T_i^j(W_t)}^2 \leq \epsilon_0 (1 + a(\epsilon))^2\norm{W_t}^2$. Also note that because $f$ is at most linear growth, there exists constant $A$ and $B$ that only depends on $x$ so that $g(y) \leq A\norm{y} + B$ for every $y$. Thus,
\begin{align*}
\frac{\E[|g(T_i^j(W_t))|]}{t}& \leq \frac{\E[\epsilon_0 (1 + a(\epsilon))^2\norm{W_t}^2]}{t} +  \frac{\E[|g(T_i^j(W_t))|1_{\set{W_t \geq \delta_0}}]}{t}\\
& \leq  \epsilon_0 (1 + a(\epsilon))^2 + \frac{\E[((1+ a(\epsilon))A\norm{W_t}+B)1_{\set{W_t \geq \delta_0}}]}{t}
\end{align*}

Then $$\limsup_{t \to 0}\frac{\E[|g(T_i^j(W_t))|]}{t}  \leq \epsilon_0 (1 + a(\epsilon))^2 + 0$$

By letting $\epsilon_0 \to 0$, we obtain the identity \cref{t_eqn}.\\

By \cref{t_eqn} and by the bound we derived for $g(W_t)$, we get:

\begin{align*}
\limsup_{t \to 0} \frac{\E[g(W_t)^+]}{t} &\leq \sum_{i =1}^N \sum_{j=1}^{2^{d-1}} \limsup_{t \to 0}\frac{\E[|g(T_i^j(W_t))|]}{t} + \frac{1}{2}a(\epsilon)\norm{Q}\\
&= \sum_{i =1}^N \sum_{j=1}^{2^{d-1}} 0 + \frac{1}{2}a(\epsilon)\norm{Q} = \frac{1}{2}a(\epsilon)\norm{Q}
\end{align*}

By letting $\epsilon \to 0$, $a(\epsilon) \to 0$, we obtain:
$$\limsup_{t \to 0} \frac{\E[g(W_t)^+]}{t}  = 0$$

Thus, by \cref{main_inequality}, we can finish the proof here.
\end{proof}
\vspace{2mm}
\section{From expectation to almost everywhere}
Now we use the Ito's lemma for convex function in \cref{ito_lemma} to show an important analytic property: a convex function is second differentiable almost everywhere. Even though Ito's lemma only gives a calculation in the form of expectation/averaging, the strong form of Ito's lemma can help us transform the expectation setting into a point-wise setting. Before going to the main \cref{alexandrov}, we first begin with two crucial lemmas, \cref{bound_sup_by_expectation} and \cref{diff_from_expectation}, that allow this transformation between two types of settings.

\begin{lem}\label{bound_sup_by_expectation}
There exists a universal constants $C > 0$ and $\alpha \in (0, 1)$ such that for nonnegative locally Lipschitz function $g$ with $g(0) = 0$, and $t = r^2 > 0$, we have
$$\sup_{B(r)} g \leq C(r\Lip_{B(r)}g)^{\alpha}\E[g(W_t)]^{1-\alpha}$$
\end{lem}
\begin{proof}
We assume that $r=1$. Let $x_0$ be the point in $\bar{B}(1)$ where $\sup g$ is attained and let $L = Lip_{B(1)}g$. Since $g(0) = 0$, we note that $L \geq g(x_0)$. The Lipschitz assumption implies that $g$ lies above the (inverted) conical region $F$ with vertex $(x_0, g(x_0))$ and a base $B(1) \cap K$, where $K = \set{x \in \mathbb{R}^n: |x-x_0| \leq g(x_0)/L} = \bar{B}(x_0, g(x_0)/L)$. Since the center of $K$ is inside $B(1)$ and its radius is at most $1$, the volume of $K \cap B(1)$ is bounded by a fraction of volume of $K$, given by $C(g(x_0)/L)^n$. The $n$-dimensional volume of $K$ is further bounded by $C(g(x_0)/L)^n g(x_0)$, and so
\begin{equation}
\int_{B(1)}g(x)dx \geq C(g(x_0)/L)^n g(x_0) \text{, i.e, } g(x_0) \leq C\left(\int_{B(1)} g(x) dx \right)^{1/(n+1)} L^{\frac{n}{n+1}}
\end{equation}
Since the Lebesgue and the distribution of $W_1$ (the unit Gaussian measure) are equivalent on $B(1)$ and $g$ is non-negative, we have $\int_{B(1)} g(x)dx \leq C \E[g(W_1)]$, for some universal $C$. Lastly, to remove the assumption $r=1$, we apply the inequality to a scaled version $x \to g(xr)$ of $g$.
\end{proof}
\vspace{2mm}

\begin{lem}\label{diff_from_expectation} (Uniform from $L^1$-convergence)
Given a convex function $f$ on $\mathbb{R}^d$ with at most polynomial growth, and consider a point $x$ so that an Ito-like condition is satisfied:
$$\lim_{t \to 0}\frac{1}{t}\E[|f(W_t^x) - f(x) - \innerprod{p(x), W_t} - \innerprod{Q(x)W_t, W_t}|]  = 0$$
Then $f$ is second differentiable at $x$ with the derivatives $(p(x), Q(x))$
\end{lem}
\begin{proof}
Without the loss of generality, we assume that $f(0) = 0$ and $p = 0$, and for $x \in \mathbb{R}^n$, $r \geq 0$, and $t \geq 0$ we define:
\begin{equation}
g(x) = |f(x) - \frac{1}{2}\innerprod{Qx, x}|, s(r) = \sup_{x \in B(r)} g(x) \text{ and } G(r) = \E[g(W_{\sqrt{r}})]
\end{equation}
so that $G(r) = o(r^2)$, as $r \to 0$. The standard estimate $\Lip_{B(r)}f \leq \frac{C}{r} \sup_{B(2r)} f$ valid for convex functions to obtain:
\begin{equation}
\Lip_{B(r)}g \leq \Lip_{B(r)}f + \Lip_{B(r)}\frac{1}{2}\innerprod{Q.,.} \leq \frac{C}{r}\sup_{B(2r)}|f| + \frac{1}{2}Cr \leq \frac{C}{r}(s(2r)+r^2)
\end{equation}
where $C$ depends only on $n$ and $Q$. \cref{bound_sup_by_expectation} implies that:
\begin{equation}\label{bound1}
s(r) \leq C(s(2r) + r^2)^{\alpha}G(r)^{1-\alpha}
\end{equation}

For $r > 0$ is so small that $C(r^{-2}G(r))^{1-\alpha} \leq \frac{1}{8}$, we have $s(r)/r^2 \leq \frac{1}{2}(s(2r)/(2r)^2+1)^\alpha$, which in turn implies that $r^{-2}s(r)$ stays bounded as $r \to 0$. We just need divide by $r^2$ and take the $\limsup_{r \to 0}$ to complete the proof.
\end{proof}
\vspace{2mm}

\begin{thm}\label{alexandrov} (Busemann-Feller-Alexandrov)
A convex function $f$ on $U \subset \mathbb{R}^d$ is second differentiable almost everywhere.
\end{thm}
\begin{proof}
Because we only focus on the differentiability, we can consider the restriction of any convex function $f$ to some bounded domain. Then we can extend that restricted function to the whole $\mathbb{R}^d$ to get a new convex function with at most linear growth. Any a.e (second) differentiability property of the new function yields the same property for the restricted function. We do the extension as follows: suppose we're given the restricted function $f$ on $\overline{B(0, 1)}$. Fix a point $z$ outside $\overline{B(0, 1)}$. We consider the set $\mathcal{Z}$ be the set of $(x, y) \in \overline{B(0, 1)}^2$ so that $x, y, z$ is on a ray that goes from the inside to the outside of the unit ball in that order. Then we extend $f$ to the whole domain $\mathbb{R}^d$ by defining:
$$f(z):= \max_{(x, y) \in \mathcal{Z}} \left( f(x) + \frac{\norm{z-x}(f(y)-f(x))}{\norm{y-x}}\right)$$

As a result, we can reduce to the case when the function $f$ is at most linear growth. The proof then follows from \cref{ito_lemma} and \cref{diff_from_expectation}.
\end{proof}
\vspace{2mm}

\begin{ack}
This work is done during 2020 at the Department of Mathematics, University of Texas at Austin. We would like to thank Professor Gordan Zitkovic and Professor Mihai Sirbu for helpful discussions.
\end{ack}

\nocite{*}
% Bibliography
\newpage
\bibliography{citation}{}

@article{barlow_protter_1990, title={On convergence of semimartingales}, DOI={10.1007/bfb0083765}, journal={Lecture Notes in Mathematics Seminaire de Probabilites XXIV 1988/89}, author={Barlow, Martin T. and Protter, Philip}, year={1990}, pages={188–193}}

@article{bourgain_sato_1986, title={A direct proof of van der Vaart's theorem}, volume={84}, DOI={10.4064/sm-84-2-125-131}, number={2}, journal={Studia Mathematica}, author={Bourgain, J. and Sato, H.}, year={1986}, pages={125–131}}

@article{carlen_protter_1992, title={On semimartingale decompositions of convex functions of semimartingales}, volume={36}, DOI={10.1215/ijm/1255987418}, number={3}, journal={Illinois Journal of Mathematics}, author={Carlen, Eric and Protter, Philip}, year={1992}, pages={420–427}}

@article{dudley_1977, title={On second derivates of convex functions.}, volume={41}, DOI={10.7146/math.scand.a-11710}, journal={Mathematica Scandinavica}, author={Dudley, R. M.}, year={1977}, pages={159}}

@article{grinberg_2013, title={Semimartingale decomposition of convex functions of continuous semimartingales by Brownian perturbation}, volume={17}, DOI={10.1051/ps/2011146}, journal={ESAIM: Probability and Statistics}, author={Grinberg, Nastasiya F.}, year={2013}, pages={293–306}}

@article{sato_1985, title={Characteristic functional of a probability measure absolutely continuous with respect to a Gaussian Radon measure}, volume={61}, DOI={10.1016/0022-1236(85)90035-7}, number={2}, journal={Journal of Functional Analysis}, author={Sato, Hiroshi}, year={1985}, pages={222–245}}

@book{evans_gariepy_1992, place={Boca Raton}, title={Measure theory and fine properties of functions}, publisher={CRC Press}, author={Evans, Lawrence C. and Gariepy, Ronald F.}, year={1992}}

@book{feller_1966, place={New York}, title={An introduction to probability theory and its applications}, publisher={Wiley}, author={Feller, W.}, year={1966}}

@book{fukushima_oshima_yoichi_takeda_1994, place={Berlin}, title={Dirichlet forms and symmetric Markov processes}, publisher={de Gruyter}, author={Fukushima, Masatoshi and Oshima Yoichi and Takeda, Masayoshi}, year={1994}}

@book{revuz_yor_1999, place={Berlin}, title={Continuous martingales and brownian motion}, publisher={Springer}, author={Revuz, Daniel and Yor, Marc}, year={1999}}
\bibliographystyle{plain} 

\end{document}